\documentclass[12pt]{amsart}

\usepackage{fullpage}
\usepackage{amsmath,amssymb,amsfonts,amsthm}

\usepackage{tikz}
\usepackage{pgfplots}
\usetikzlibrary{arrows}
\usepackage[colorinlistoftodos]{todonotes}
\usepackage[colorlinks=true, allcolors=blue]{hyperref}

\newtheorem{theorem}{Theorem}[]

\newtheorem{lemma}[theorem]{Lemma}

\newtheorem{corollary}[theorem]{Corollary}

\theoremstyle{remark}
\newtheorem{remark}[theorem]{Remark}

\begin{document}
\title[Scaled Oscillation and Level Sets]{Scaled Oscillation and Level Sets}

\author{Iqra Altaf}
 \address{Department of Mathematics, The University of Chicago, 5734 S. University Avenue, Chicago, IL 60637, USA. email:
iqra@math.uchicago.edu.}
\thanks{}
 \author{Marianna Cs\"ornyei}
 \address{Department of Mathematics, The University of Chicago, 5734 S. University Avenue, Chicago, IL 60637, USA. email:
csornyei@math.uchicago.edu.}
\thanks{}
\author{Bobby Wilson }
\address{University of Washington, Department of Mathematics, Box 354350,  Seattle, WA 98195-4350, USA. email:
blwilson@uw.edu.}
\thanks{B. Wilson was  supported by NSF  grant  DMS 1856124.}



\date{\today} % clear date

\begin{abstract}
We study the size and regularity properties of level sets of continuous functions with bounded upper-scaled and lower-scaled oscillation.
\end{abstract}
\maketitle

\section{Introduction}
Let $f: A\rightarrow \mathbb{R}^k$, $A \subset [0, 1]^m$, be a Lipschitz continuous function. According to Rademacher's theorem, a Lipschitz function is differentiable almost everywhere in its domain. In fact, Stepanov's theorem implies that it suffices to assume that the upper-scaled oscillation (defined below) of a function is finite almost everywhere in order to conclude that it is differentiable almost everywhere.

        Assuming $m \geq k$, another well-known regularity property of a Lipschitz function, $f :A\rightarrow \mathbb{R}^k$, is that for  Lebesgue almost every $z \in \mathbb{R}^k$,  the level set $f^{-1}(z)$ is countably $(m-k)$-rectifiable and $\mathcal{H}^{m-k}(f^{-1}(z)) < \infty$. In fact, it suffices to prescribe that 
 the graph of $f$ has finite ${\mathcal{H}}^{m-k}$ measure then almost all of its level sets have finite ${\mathcal{H}}^{m-k}$ measure. For the more regular class of $C^n$ functions one has the stronger theorem, known as Sard's theorem, which states that if $f : \mathbb{R}^{m} \rightarrow \mathbb{R}^{k}$ is $n$-times continuously differentiable for $n\ge \max\{m-k+1,1\}$ and 
 	\begin{align*}
 		S= \{ x \in \mathbb{R}^{m}: \mbox{rank}(\nabla f(x))<k\},
	\end{align*}
	then $f(S)$ has Lebesgue measure zero.

   In order to demonstrate necessary and sufficient conditions on the vector field of the continuity equation that would ensure uniqueness of solutions, Alberti, Bianchini, and Crippa \cite{alberti2014uniqueness} define a notion of a ``weak Sard" property. Specifically, they showed in \cite{alberti2014uniqueness} that if the vector field, $b:[0, T] \times \mathbb{R}^2 \rightarrow \mathbb{R}^2$, is divergence-free and  bounded then the  following Cauchy problem
	\begin{align*}
		\left\{\begin{array}{l}
		\partial_t u = -\mbox{div}(bu) \\ 
				u(0)=u_0 \in W^{1, \infty}(\mathbb{R}^2) 
				\end{array}\right.
	\end{align*}
	has a unique weak solution in $C([0, T]; L^{\infty})$ if and only if the potential of $b$, $f$ (defined by satisfying $b= \nabla^{\perp} f$), satisfies the weak Sard property. The function $f$ is said to satisfy the weak Sard property if any measure $\mu$, supported on $f(S\cap E)$, is mutually singular to the Lebesgue measure on $\mathbb{R}^k$, where $E$ is ``the union of all connected components with positive
        length of all level sets." Moreover, in \cite{ABC}, Alberti, et. al. show that the minimal regularity for the weak Sard property to hold is $f \in W^{2, 1}(\mathbb{R}^2)$. Initially, Alberti \cite{alberti2012generalized}, taking into account that $W^{2, 1}$ functions are not necessarily differentiable, proved that the size of the image of $S \cup N$, where 
        \begin{align*}
        		N = \{ x\in  \mathbb{R}^2 : \nabla f \mbox{ does not exist}\},
        	\end{align*} 
        	is Lebesgue negligible.  Further analysis discussed in \cite{alberti2014uniqueness} and \cite{ABC} demonstrates that more regularity for $f$ is necessary to achieve equivalent results for the higher-dimensional case, $f: \mathbb{R}^m \rightarrow \mathbb{R}^{m-1}$. Specifically, the authors show that one needs $f \in C^{1, 1/2}$.

        In this paper, we consider the characterization of $f(S \cup N)$ and the characterization of the level sets of $f$ when $f$ satisfies the following ``local Lipschitz" conditions. For $x \in A$, the  upper-scaled and lower-scaled oscillation of $f$ are defined as
	\begin{align*}
		L_f(x) &= \limsup_{r \rightarrow 0}\frac{{\sup_{d(x,y)\leq r}}|f(y)-f(x)|}{r} \hspace{1cm} \mbox{ and}\\
		l_f(x) &= \liminf_{r \rightarrow 0}\frac{{\sup_{d(x,y)\leq r}}|f(y)-f(x)|}{r}
  	\end{align*}
        respectively.
The differentiability properties of functions satisfying either $L_f(x) <\infty$ or $l_f(x)<\infty$ on some suitably sized subset of the domain of $f$ were studied initially in \cite{balogh2006scaled} and then further studied in \cite{csornyei2016tangents}. In \cite{balogh2006scaled}, the authors show that $L^p$ boundedness of the function, $l_f$, implies Sobolev regularity for the function, $f$. Moreover, one simply needs the boundedness condition, $l_f(x)<\infty$, in order to prove the existence of a set of positive measure on which the function, $f: [0, 1] \rightarrow \mathbb{R}$, is differentiable. In \cite{csornyei2016tangents}, among other results, the authors show that for higher-dimensional domains, $f: [0, 1]^m \rightarrow \mathbb{R}$, $l_f(x) < \infty$ implies the existence of positive measure sets of differentiability for $f$.

One can ask to what extent the weak Sard property fails for functions with such low regularity. In consideration of this question, we construct continuous functions with finite lower-scaled oscillation at each $x$ whose level sets are ``large''. First, we present our result in dimension 1. 

\begin{theorem}\label{OneD}
	There exists a continuous function $f:[0,1] \rightarrow[0,1]$ such that $l_{f}(x)<\infty$ for all $x \in[0,1]$ and $f^{-1}(y)$ is infinite for every $y \in[0,1]$.
\end{theorem}

It is clear that for every $y \in [0, 1]$ in the above construction, $f^{-1}(y) \cap (N \cup S) \neq \emptyset$. One should also note that,  due to the argument in \cite{csornyei2016tangents} demonstrating regularity of such functions, countably infinite is the most that should be expected for almost every level set. For the higher dimensional case, we obtain the following result:

\begin{theorem}\label{HighD}
	 For $m \geq 2$ there exists a continuous function $f:[0,1]^{m} \rightarrow[0,1]$ such that $l_{f}(x)<\infty$ for all $x \in[0,1]^{m}$, and $f^{-1}(y)$ is not $(m-1)$-rectifiable for every $y \in[0,1]$.
\end{theorem}

It is important to emphasize that the level sets in the above construction are not purely unrectifiable although they are unrectifiable. In fact, the existence of positive measure subsets on which such functions are differentiable,  \cite{csornyei2016tangents}, implies that generic level sets will have rectifiable parts. 

Trivially, if we take the construction, $f: [0, 1]^m \rightarrow [0,1]$, from Theorem \ref{OneD} or Theorem \ref{HighD} and construct the function, $g: [0, 1]^{\ell+m}  \rightarrow [0, 1]^{\ell+1}$, defined by $g(z, x):= (z, f(x))$ for $x \in [0, 1]^m$ and $z \in [0, 1]^{\ell}$, we have the following corollary:

	\begin{corollary}
          \begin{enumerate}
            \item For $m=k\ge 1$, there exists a continuous function $f:[0,1]^{m} \rightarrow[0,1]^k$ such that $l_{f}(x)<\infty$ for all $x \in[0,1]^{m}$, and $f^{-1}(y)$ is infinite for all  $y \in[0,1]^k$.
            \item For $m>k\geq 1$, there exists a continuous function $f:[0,1]^{m} \rightarrow[0,1]^k$ such that $l_{f}(x)<\infty$ for all $x \in[0,1]^{m}$, and $f^{-1}(y)$ is not $(m-k)$-rectifiable for all  $y \in[0,1]^k$.
            \end{enumerate}
          \end{corollary} 

Regarding functions with $L_{f}(x)< \infty$ for every $x \in \mathbb{R}^m$,
% $m\geq 2$,
one can divide $\mathbb{R}^m$ into countably many parts, on each of which the function is $L$-Lipschitz for some $L$. Thus almost all level sets of such functions are rectifiable and have $\sigma$-finite ${\mathcal{H}}^{m-1}$ measure. However, there is a discrepancy between the characterization of level sets of functions with finite upper-scaled oscillation defined on $\mathbb{R}$, and in higher dimension. A simple construction (Remark \ref{Lfexample}) demonstrates that the level sets of functions, $f: [0, 1]^m \rightarrow \mathbb{R}$, satisfying  $L_{f}(x)< \infty$ for all $x \in [0,1]^m$ may have infinite $\mathcal{H}^{m-1}$ measure, while almost every level set must be finite when $m=1$. 

The structure of the paper is as follows: we begin with a discussion of the preliminary tools necessary to analyze functions with finite upper/lower scaled oscillation. In this preliminary section, we will also discuss our simple results concerning functions with finite upper-scaled oscillations. In the main part of our paper we will construct functions with finite lower-scaled oscillation that will prove Theorems \ref{OneD} and \ref{HighD}.
%%%%%%%%%%%%%%%%%%%%%%%%%%%%%%%%%%%%%%%%%%%%%%%%%%%%%%%%%%%%%%%%
%%%%%%%%%%%%%%%%%%%%%%%%%%%%%%%%%%%%%%%%%%%%%%%%%%%%%%%%%%%%

\section{Preliminaries}
Throughout the paper, we denote by $B(x,r) \subset \mathbb{R}^m$, the open ball of radius $r$ and center $x$. The Lebesgue measure on $\mathbb{R}^m$ is denoted by $\mathcal{L}^m$, and the $s$-dimensional Hausdorff measure is denoted by ${\mathcal{H}}^s$.
	 Recall that the density of $A$ in the ball $B(x,r)$ is defined by
	\begin{align*}
		{D}^{s}(A,x,r)=\frac{{{\mathcal{H}}^s(A\cap B(x,r))}}{(2r)^s},
	\end{align*}
	and the $s$-density of $A$ at the point $x$ is defined by
	\begin{align*}
		{D}^{s}(A,x)=\lim_{r \rightarrow 0} {D}^{s}(A,x,r)
	\end{align*}
if the limit exists.

A set $A \subset \mathbb{R}^m$ is called $s$-rectifiable if there exist $f_i :\mathbb{R}^{s}\rightarrow \mathbb{R}^{m}, i \in \mathbb{N}, $ such that 
	\begin{align*}
		{\mathcal{H}}^s\big(A\setminus \bigcup_{i=1}^{\infty}f_i(\mathbb{R}^s)\big)=0.
	\end{align*}
An unrectifiable set is a set which is not rectifiable. One characterization of rectifiability of sets is the density theorem of Marstrand and Mattila:

\begin{theorem}[Theorem 16.2 from \cite{mattila1999geometry}] \label{densitythm}
Let $A$ be an $ {\mathcal{H}}^s$ measurable subset of $\mathbb{R}^m$ such that ${\mathcal{H}}^s(A) < \infty.$ Then $A$ is $s$-rectifiable if and only if $D^{s}(A,x)=1$ for $ {\mathcal{H}}^s$ almost all $x\in A.$
\end{theorem}

 Let $f : \mathbb{R}^m \rightarrow \mathbb{R}$ be a function and let $E$ be a subset of $\mathbb{R}^m$. We denote the graph of $f$ on $E$ by $G_{f}(E).$ 
	
	\begin{lemma} \label{scalelevel}
Let $f:[0,1]^m\rightarrow \mathbb{R}$ be a continuous function such that $l_{f}(x) < \infty$, for all $x\in [0,1]$. Then \\
(i) $G_{f}([0,1])$ has $\sigma$-finite ${\mathcal{H}}^{m}$-measure. \\
(ii) ${\mathcal{H}}^{m}\big(G_{f}(H)\big)=0$ for every set $H \subset [0,1]^m$ with $ {\mathcal{H}}^{m}(H)=0$. 
	\end{lemma}

We refer to \cite{csornyei2016tangents} for the proof of this statement.

As discussed in the introduction, it follows from part (i) of Lemma \ref{scalelevel} that if $f:[0,1]^{m} \rightarrow \mathbb{R}$ is a function such that $l_{f}(x)< \infty$ for all $x \in [0,1]^m$ then almost all level sets of $f$ have $\sigma$-finite ${\mathcal{H}}^{m-1}$ measure. For $m=1$, part (ii) of Lemma \ref{scalelevel} implies that $f$ satisfies the Lusin's condition. This leads to the following result.

	\begin{theorem} \label{BigLipLevelSets}
		Let $f: [0,1]\rightarrow \mathbb{R}$ be a continuous function such that $L_{f}(x)< \infty$ for all except countably many $x\in [0,1]$. Then  $f^{-1}(y)$ is finite for almost every $y \in \mathbb{R}$. 
	\end{theorem}

This statement follows by first observing that if a level set were to have infinitely many points then it has at least one accumulation point. At this point the function either has a zero derivative or the derivative does not exist. Since the function is differentiable almost everywhere, Lusin's condition implies that the image of the set of points at which the function is not differentiable is measure zero. Furthermore, a standard argument implies that the set of points at which the derivative is zero also has null image. 

However, if we replace the countable exceptional set in Theorem \ref{BigLipLevelSets} with other exceptional sets, we cannot conclude that almost all level sets are finite. In fact, as described in the following remark, countability of the exceptional set is necessary.
	\begin{remark} \label{BigLipRemark}
		Let $E$ be an uncountable Borel set. There exists a continuous function $f:[0,1]\rightarrow [0,1]$ such that $L_{f}(x) < \infty$ for all $x \in [0,1]\setminus E $ and the level set $f^{-1}(y)$ has infinite cardinality for each $y \in [0,1]$.
	\end{remark}
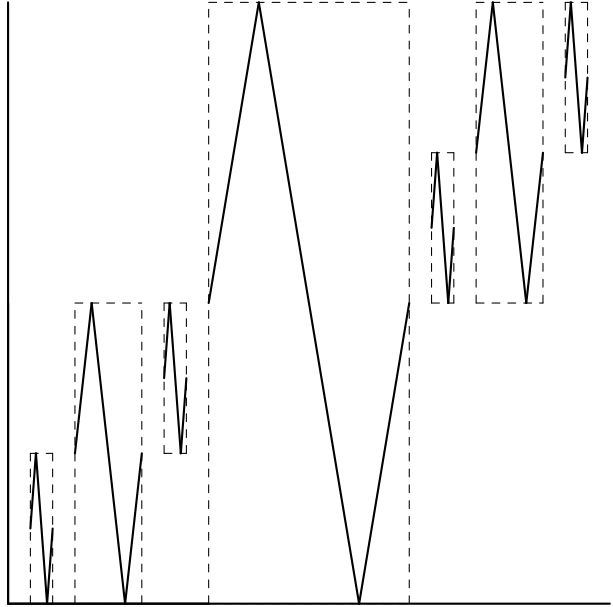
\begin{center}
\begin{tikzpicture}[ line cap=round, line join=round ]
\begin{scope}[yscale = 8, xscale = 8]
\draw[thick] (0,0) -- (1,0) ;
	\draw[thick] (1/3,1/2) -- (1/3+1/12,1) ;
	\draw[thick] (2/3-1/12,0) -- (1/3+1/12,1) ;
	\draw[thick] (2/3-1/12,0) -- (2/3,1/2) ;
	\draw[thick] (0,0) -- (0,1) ;
	\draw[dashed] (1/3,0)--(1/3,1);
	\draw[dashed] (2/3,0)--(2/3,1);
	\draw[dashed] (1/3,1)--(2/3,1);
	\draw[dashed] (1/3,0)--(2/3,0);
	\end{scope}
	
\begin{scope}[yscale = 8/2, xscale = 8/3, shift={(0, 0)}]
\draw[thick] (0,0) -- (1,0) ;
	\draw[thick] (1/3,1/2) -- (1/3+1/12,1) ;
	\draw[thick] (2/3-1/12,0) -- (1/3+1/12,1) ;
	\draw[thick] (2/3-1/12,0) -- (2/3,1/2) ;
	\draw[thick] (0,0) -- (0,1) ;
	\draw[dashed] (1/3,0)--(1/3,1);
	\draw[dashed] (2/3,0)--(2/3,1);
	\draw[dashed] (1/3,1)--(2/3,1);
	\draw[dashed] (1/3,0)--(2/3,0);
	\end{scope}
	
\begin{scope}[yscale = 8/2, xscale = 8/3, shift={(2, 1)}]

	\draw[thick] (1/3,1/2) -- (1/3+1/12,1) ;
	\draw[thick] (2/3-1/12,0) -- (1/3+1/12,1) ;
	\draw[thick] (2/3-1/12,0) -- (2/3,1/2) ;
	
	\draw[dashed] (1/3,0)--(1/3,1);
	\draw[dashed] (2/3,0)--(2/3,1);
	\draw[dashed] (1/3,1)--(2/3,1);
	\draw[dashed] (1/3,0)--(2/3,0);
	\end{scope}
	
	\begin{scope}[yscale = 8/4, xscale = 8/9]

	\draw[thick] (1/3,1/2) -- (1/3+1/12,1) ;
	\draw[thick] (2/3-1/12,0) -- (1/3+1/12,1) ;
	\draw[thick] (2/3-1/12,0) -- (2/3,1/2) ;
	
	\draw[dashed] (1/3,0)--(1/3,1);
	\draw[dashed] (2/3,0)--(2/3,1);
	\draw[dashed] (1/3,1)--(2/3,1);
	\draw[dashed] (1/3,0)--(2/3,0);
	\end{scope}
	
	\begin{scope}[yscale = 8/4, xscale = 8/9, shift={(2, 1)}]

	\draw[thick] (1/3,1/2) -- (1/3+1/12,1) ;
	\draw[thick] (2/3-1/12,0) -- (1/3+1/12,1) ;
	\draw[thick] (2/3-1/12,0) -- (2/3,1/2) ;
	
	\draw[dashed] (1/3,0)--(1/3,1);
	\draw[dashed] (2/3,0)--(2/3,1);
	\draw[dashed] (1/3,1)--(2/3,1);
	\draw[dashed] (1/3,0)--(2/3,0);
	\end{scope}
	
	\begin{scope}[yscale = 8/4, xscale = 8/9, shift={(8, 3)}]

	\draw[thick] (1/3,1/2) -- (1/3+1/12,1) ;
	\draw[thick] (2/3-1/12,0) -- (1/3+1/12,1) ;
	\draw[thick] (2/3-1/12,0) -- (2/3,1/2) ;
	
	\draw[dashed] (1/3,0)--(1/3,1);
	\draw[dashed] (2/3,0)--(2/3,1);
	\draw[dashed] (1/3,1)--(2/3,1);
	\draw[dashed] (1/3,0)--(2/3,0);
	\end{scope}
	
	\begin{scope}[yscale = 8/4, xscale = 8/9, shift={(6, 2)}]

	\draw[thick] (1/3,1/2) -- (1/3+1/12,1) ;
	\draw[thick] (2/3-1/12,0) -- (1/3+1/12,1) ;
	\draw[thick] (2/3-1/12,0) -- (2/3,1/2) ;
	
	\draw[dashed] (1/3,0)--(1/3,1);
	\draw[dashed] (2/3,0)--(2/3,1);
	\draw[dashed] (1/3,1)--(2/3,1);
	\draw[dashed] (1/3,0)--(2/3,0);
	\end{scope}
\end{tikzpicture} 

    $$\textbf{Construction of the function when $E$ is the middle third cantor set.}   $$
\end{center}
In order to prove the remark, we first consider the case where $E$ is the middle-third Cantor set and consider the Cantor function over that set. For each of the $2^{n-1}$ intervals of size $3^{-n}$ in the complement of the Cantor set we replace the flat piece of the function defined on this interval with a piecewise linear function whose image is an interval of length $2^{-(n-1)}$ so that the intervals of this size cover the interval $[0, 1]$.  One can do this in a continuous way.

For an arbitrary Cantor set $E$, we can find a homeomorphism $h:\,\mathbb{R}\,\to \mathbb{R}$ that maps $E$ to the middle-third Cantor set, and which is linear on each component of the complement of $E$, and then compose this homeomorphism $h$ with the function $f$ we have constructed for the middle-third Cantor set. Finally, the statement of Remark \ref{BigLipRemark} follows since every uncountable Borel set contains a Cantor set.

\medskip

As discussed in the introduction, the characteristics of level sets of functions satisfying $L_f(x)<\infty$ is  heavily determined by the dimension of the domain. Although $L_f(x)<\infty$ for every $x$ guarantees that almost every level set will have finite cardinality in the one-dimensional case, in the two-dimensional case and higher, there is no such phenomenon as displayed in the following example.

\begin{remark} \label{Lfexample}
		There exists a continuous function $f:[0,1]^m\rightarrow [0,1]$ such that $L_{f}(x) < \infty$ for all $x \in [0,1]^m $ and the level sets $f^{-1}(y)$ have infinite ${\mathcal H}^{m-1}$-measure for all $y \in [0,1]$.
	\end{remark}
	
	\begin{proof}[Proof of Remark \ref{Lfexample}]
		It suffices to construct an example for $m=2$.  Consider the function $g: [0, 1] \rightarrow [0, 1]$ defined by 
			\begin{align*}
				g(x) = \left\{ \begin{array}{ll}  x\sin^2(x^{-1}) & x \in (0, 1]\\
									0 & x=0.   \end{array} \right.
			\end{align*}
		It is clear that $L_g(x) <\infty$ for all $x \in [0,1]$  and the graph of $g$ has infinite $\mathcal{H}^1$ measure. Now define 
			\begin{align*}
				f(x,y) = \left\{  \begin{array}{ll} 1 & y-g(x)\geq 1\\
								0 & y-g(x) \leq 0\\
								y-g(x) & \mbox{otherwise.}
								 \end{array} \right.
			\end{align*}
			Then for every $z \in [0,1]$, $f^{-1}(z)= \{ (x, y) \in [0,1]~:~ y=g(x) -z\}$ so the level sets have infinite length. Finally, 
			\begin{align*}
				f(x_1, y_1)-f(x_2, y_2) = y_1-y_2- (g(x_1)-g(x_2))
			\end{align*} 
			and thus $L_f(x,y) \leq 2\max( 1, L_g(x)) < \infty$ for all $(x, y)$.
	\end{proof}

%
%We will use the following Whitney decomposition theorem (see e.g, J.1 in Grafakos, [4]). We denote by $l(Q)$ the side length of a cube $Q$ in $\mathbb{R}^m$.
%
%\begin{thm} Let $\Omega $ be a bounded open nonempty subset of $\mathbb{R}^m$. Then there exists a family of closed dyadic cubes $\{Q_{i}\}_{i \in \mathbb{N}}$ (called the Whitney cubes of $\Omega$) with the following properties. \\
%(i)  $\Omega=\bigcup_{i=1}^{\infty}Q_{i}$ and $Q_{i}$'s have disjoint interiors. \\
%(ii)  $l(Q_{i})\leq \dist(Q_{i},\Omega^{c})\leq 4l(Q_{i})$. \\
%\end{thm}

%%%%%%%%%%%%%%%%%%%%%%%%%%%%%%%%%%%%%%%%%%%%%%%%%%%%%%%%%%%%%%
%%%%%%%%%%%%%%%%%%%%%%%%%%%%%%%%%%%%%%%%%%%%%%%%%%%%%%%%%%%%%%

\section{Lower-Scaled Oscillation}

In this section, we present the proofs to Theorems \ref{OneD} and \ref{HighD}. 

\begin{proof}[Proof of Theorem \ref{OneD}]
%	It is sufficient to prove the claim of Theorem \ref{OneD} and \ref{HighD} for all level sets $f^{-1}(y)$ except countably many. A combination of two functions with disjoint domains and exceptional sets then proves Theorem \ref{OneD} and \ref{HighD}. 

  Our construction of the function is based on an iterative argument on rectangles. %using Whitney decompositions (see e.g, J.1 in Grafakos, \cite{grafakos2008classical}).  We denote by $l(Q)$ the side length of a cube $Q$ in $\mathbb{R}^m$.
  Assume we have a collection of rectangles with sides parallel to the coordinate axes. We will divide each rectangle into two parts vertically. We will define the function piece-wise linearly on the right rectangle. We will then divide the diagonal of the left rectangle into countably many line segments and consider the rectangles with these line segments as diagonals and sides parallel to the coordinate axes. These will form a new collection of rectangles. We will explain the details below.

  Let $\mathbb{P}_{x}$ and $\mathbb{P}_{y}$ be the projection maps onto the $x$-axis and $y$-axis respectively. For a rectangle $Q$ with sides parallel to the coordinate axes, we denote by $l(Q)$ and $h(Q)$ the length of its horizontal and vertical sides, respectively. We start with the square $Q_{0, 1}:=[0,1]^2$ and the function $f_0(x)=x$.
%and $E_{1}=\emptyset$.

Assume that we have completed $n\in \mathbb{N}$ steps of the construction, and we have (at most) countably many rectangles $(Q_{n,i})_{i \in \mathbb{N}}$ and a function $f_n$, such that inside each $Q_{n,i}$ the graph of $f_n$ is the linear line segment connecting its bottom left and its top right corners. We will define $f_{n+1}$ by replacing this line segment inside each $Q_{n,i}$ by another continuous graph. Outside $\bigcup_{i=1}^{\infty}\mathbb{P}_{x}(Q_{n,i})$ we define $f_{n+1}=f_n$. %Moreover, assume that $f$ can be extended continuously to $[0, 1]$ by linearly connecting the diagonal vertices of the $Q_{n, i}$. Finally, assume there exists some countable set $E$ (to be described later) such that  $|\{ x  \in E_{n}: f(x)=y\}|=2n$ for every $y \in [0,1] \setminus E$ at the $n$th step of the construction.

Consider one of the rectangles $Q=Q_{n,i}$, and denote $Q=[x,x+l(Q)]\times [y,y+h(Q)]$.
% be represented by $Q_{n, i}=[x_{n,i}, x_{n,i}+l(Q_{n,i})] \times [y_{n,i}, y_{n,i}+h(Q_{n,i})]$. The definitions of the functions, $h$ and $l$, follow clearly from this description of $Q_{n, i}$.
We divide $Q$ vertically into two rectangles,
%$Q'=Q'_{n,i}$ and $Q''=Q''_{n,i}$,  
	\begin{align*}
		Q'=[x, x+(1-a_n)l(Q)] \times [y, y+h(Q)]
	\end{align*} 
	and 
	\begin{align*}
		Q''=[x+(1-a_n)l(Q),x+l(Q)] \times [y, y+h(Q)],
	\end{align*}
        where $0<a_n\ll1$ is to be chosen later. In $Q''$, we define $f_{n+1}$ to be an arbitrary piecewise linear function connecting its top left and top right corners, such that the graph of the function lies inside $Q''$, and it has a horizontal piece lying on the top and on the bottom side of $Q''$, respectively. Therefore the level sets $f_{n+1}^{-1}(y)$ and $f_{n+1}^{-1}(y+h(Q))$ are infinite, and the pre-image of each point in the open interval $(y,y+h(Q))$ has at least two points.

%Define%
%	\begin{align*}
%		f(x_{n, i}) &=y_{n, i} \\
%		f(x_{n, i}+l(Q_{n, i})(1-a_{n})) &=y_{n, i}+h(Q_{n, i}) \\
%		f(x_{n, i}+l(Q_{n, i})(1-a_{n} / 2)) &=y_{n, i} \\
%		f(x_{n, i}+l(Q_{n, i})) &=y_{n, i}+h(Q_{n, i}).
%	\end{align*}
%Extend $f$ linearly from $x_{n,i}+l(Q_{n,i})(1-3a_n/4) $ to $x_{n,i}+l(Q_{n,i})(1-a_n/2)$ and from  $x_{n,i}+l(Q_{n,i})(1-a_n/4)$ to $x_{n,i}+l(Q_{n,i})$. 

        In $Q'$, we define $f_{n+1}$ such that its graph is the diagonal of $Q'$ connecting its bottom left and top right corners. Then we apply the 1-dimensional Whitney decomposition theorem (see e.g, J.1 in Grafakos, \cite{grafakos2008classical}) to divide this diagonal into countably many line segments, such that the length of each of these line segment is comparable to its distance from the closest endpoint of the diagonal of $Q'$.
   % apply the one-dimensional Whitney decomposition theorem to the diagonal of $Q'_{n,i}$ joining $(x_{n,i},y_{n,i})$ and $(x_{n,i}+l(Q_{n,i})(1-a_n),y_{n,i}+h(Q_{n,i}))$ to divide it into line segments.
Corresponding to each line segment, we consider a rectangle with the line segment as its diagonal and sides parallel to the coordinate axes. We call these rectangles the child rectangles of $Q_{n,i}$.
 
We repeat this construction for each $Q_{n,i}$, $i\in \mathbb{N}$ and enumerate the countably many rectangles thus obtained as $\{Q_{n+1,j}\}_{j \in \mathbb{N}}$. This completes the $(n+1)th $ step. 

  \begin{center}
\begin{tikzpicture}[yscale = 8, xscale = 8, line cap=round, line join=round ]
	% Axes
	\draw[thin] (0,0) -- (.9,0) ;
	%\draw[->] (1,0) --  node[right,above]{} ;
	\draw[thin] (.9,0) -- (.95,0);
	\draw[thin] (.95,0) -- (1.1,0); 
	\draw[thin] (0,0) -- (0,1) ;
	\draw[very thick] (.9,1)  -- (.93,1);
	\draw[very thick] (1.06,1)  -- (1.1,1);
	\draw[very thick] (.96,0)  -- (1.02,0);
	\draw[very thick] (.93,1)  -- (.96,0) node[right,black]{} ;
	\draw[very thick] (1.02,0)  -- (1.06,1) node[right,black]{} ;
	\draw[thick] (.9,-.04)  -- (1,-.04) node[below]{\footnotesize $(1-a_n)l(Q)$} ;
	\draw[thick] (.9,-.04)  -- (.9,-.02);
	\draw[thick] (1.1,-.04)  -- (1.1,-.02);
	\draw[thick] (1,-.04)  -- (1.1,-.04)  ;

%    % Box around x = 0.5 with width 2*0.2
  %    %  \draw[thick][domain = 0.93:0.95013, smooth,variable=\x] plot ({\x},{-50*\x+47.5});
 %    %   \draw[thick][domain = 0.97:1, smooth,variable=\x] plot ({\x},{(100/3)*\x-97/3});
%    %   \draw[thick, draw=black] (0.95,0) -- (0.97, 0);
 %    %  \draw[thick, draw=black] (0.9,1) -- (0.93, 1);

 % Box around x = 0.5 with width 2*0.2

    \draw[dashed] (0.9,0) -- (0.9, 1);
    \draw[thin] (1.1,0) -- (1.1, 1) ;
    \draw[thin] (0,1) -- (1.1, 1);
    \draw[very thick] (0,0) -- (.9, 1);
    
    \draw[thin] (0.3,1/3) -- (0.3, 2/3);
    \draw[thin] (0.6,1/3) -- (0.6, 2/3);
    \draw[thin] (0.3,1/3) -- (0.6, 1/3);
    \draw[thin] (0.3,2/3) -- (0.6, 2/3);
    
    \draw[thin] (0.15,1/6) -- (0.15, 1/3);
    \draw[thin] (0.3,1/6) -- (0.3, 1/3);
    \draw[thin] (0.15,1/3) -- (0.3, 1/3);
    \draw[thin] (0.15,1/6) -- (0.3, 1/6);
    
    \draw[thin] (0.6,2/3) -- (0.75, 2/3);
    \draw[thin] (0.75,2/3) -- (0.75, 5/6);
    \draw[thin] (0.6,2/3) -- (0.6, 5/6);
    \draw[thin] (0.6,5/6) -- (0.75, 5/6);
    
    \draw[thin] (0.75,5/6) -- (0.75, 85/90);
    \draw[thin] (0.75,5/6) -- (0.85, 5/6);
    \draw[thin] (0.85,5/6) -- (0.85, 85/90);
    \draw[thin] (0.75,85/90) -- (0.85, 85/90);
    
     \draw[thin] (0.05,1/18) -- (0.05, 1/6);
    \draw[thin] (0.15,1/18) -- (0.15, 1/6);
    \draw[thin] (0.05,1/18) -- (0.15, 1/18);
    \draw[thin] (0.05,1/6) -- (0.15, 1/6);
    
    \draw[thin] (0.05,1/18) -- (0.05, 1/6);
    \draw[thin] (0.15,1/18) -- (0.15, 1/6);
    \draw[thin] (0.05,1/18) -- (0.15, 1/18);
    \draw[thin] (0.05,1/6) -- (0.15, 1/6);
     \draw[thin] (0.01,1/18) -- (0.05, 1/18);
    \draw[thin] (0.01,.1/9) -- (0.05, .1/9);
    \draw[thin] (0.05,.1/9) -- (0.05, 1/18);
    \draw[thin] (0.01,.1/9) -- (0.01, 1/18);
    
    \draw[thin] (0.85,85/90) -- (0.89, 85/90);
    \draw[thin] (0.85,89/90) -- (0.89, 89/90);
    \draw[thin] (0.85,85/90) -- (0.85, 89/90);
    \draw[thin] (0.89,85/90) -- (0.89, 89/90);
    \draw (0.45,0) +(0.45cm, 0) node[below] {};
    \draw (0.51,0) +(0.51 cm, 0) node[below] {};
    \draw (0.0,0) +(0.0 cm, 0) node[below]{\footnotesize $(x,y)$};
    \end{tikzpicture} 
    $$\textbf{Figure 1}:  (n+1)^{th} \text{ step on } Q_{n,i} $$
\end{center}

It is clear from our construction that the functions $f_n$ converge to a continuous function $f$, and that $f$ takes every value $y_0$ between 0 and 1 at infinitely many points. Indeed, if $y_0$ is the $y$ coordinate of one of the vertices of one of the rectangles $Q$ of our construction, then it takes this value at infinitely many points in its right rectangle $Q''$. And if it is not the $y$ coordinate of any of the vertices of any of the rectangles, then for each $n$ we can choose a $Q_{n,i}$ s.t. $y_0\in\mathbb{P}_y(Q_{n,i})$. Then $f^{-1}(y_0)$ has at least two points in $\mathbb{P}_x(Q''_{n,i})$. Since this is true for each $n$, and the sets $\mathbb{P}_x(Q''_{n,i})$ are pairwise disjoint, therefore $f^{-1}(y_0)$ is infinite.

%Because $f$ can be extended continuously to $[0, 1]$ by linearly connecting the diagonal vertices of the $Q_{n, i}$, the current construction of $f$ on the $Q''_{n, i}$ will maintain the continuity on $\mathbb{P}_{x}(\bigcup_{i=1}^{\infty}Q''_{n,i})\cup \left([0,1]\setminus \mathbb{P}_{x}(\bigcup_{i=1}^{\infty}Q_{n,i})\right)$. Now $f$ can be extended continuously to $[0, 1]$ by linearly connecting the diagonal vertices of the $Q_{n+1, i}$.
 
%Let $x\in [0,1]\setminus \bigcup_{i=1}^{\infty}\mathbb{P}_{x}(Q''_{n,i})$. We define $f(x)$ to be the intersection of all the rectangles $Q_{n,i}$ whose projection onto the $x$-axis contains $x$. Thus the function is continuous on $[0,1]$.\par  
%
%Let $E=\{\mathbb{P}_{y}(x,y): (x,y) \text{ is the vertex of }Q_{n,i} \text{  for} \text{ some  }  n,i \in \mathbb{N}\}.$ For all $y \in [0,1]\setminus E$ and all $n \in \mathbb{N}$, $f^{-1}(y)$ intersects the interior of $\mathbb{P}_{x}(Q''_{n,i})$ at two points for some $i\in \mathbb{N}$. Since the interiors of $\mathbb{P}_{x}(Q''_{n,i})$ are disjoint for all $i, n \in \mathbb{N}$  and $|\{ x  \in E_{n}: f(x)=y\}|=2n$, we have that for $E_{n+1}= E_n \cup \mathbb{P}_{x}(\bigcup_{i=1}^{\infty}Q''_{n,i})= [0, 1] \setminus \mathbb{P}_{x}\left(\bigcup_{i=1}^{\infty}Q_{n+1,i})\right)$
%	\begin{align*}
%		|\{ x  \in E_{n+1}: f(x)=y\}|=2(n+1).
%	\end{align*}
%Continuing this construction for $n \rightarrow \infty$, we observe that for all $y \in [0,1]$, $f^{-1}(y)$ is infinite.

It remains to prove that the lower-scalled oscillation $l_{f}(x)$ is finite for each $x \in [0,1]$.  The restriction of $f$ onto the interval $\mathbb{P}_x(Q _{n,i}'')$ is Lipschitz because $f$ is piecewise linear inside each rectangle $Q_{n,i}''$.

Inside each of the rectangles  $Q_{n,i}'$, the graph of $f$ is covered by its child rectangles. We observe that for each fixed $n \in \mathbb{N}$, all the rectangles $Q_{n,i}$ are similar. Moreover the  height/length ratio at the $(n+1)th$ step increases by a factor of $1/(1-a_{n})$. Thus for all the rectangles, the  height/length ratio is bounded by $\prod_{i=0}^{\infty}(1-a_{n})^{-1}$. We choose $a_n$ small enough such that $\prod_{i=0}^{\infty}(1-a_{n})^{-1}< \infty$. \par Then selecting the diagonals of the rectangles by applying the Whitney decomposition theorem implies that there is a constant $c$ satisfying the following: if $(x,f(x))$ is one of the vertices of $Q_{n,i}'$, and $(x',f(x'))$ is an interior point of $Q_{n,i}'$, then $|f(x)-f(x')|\le c|x-x'|$. This shows that the lower scaled oscillation is finite at $x$. Putting $r:=|x'-x|$, it also implies that for any $x''\in \mathbb{P}_x (Q_{n,i}')$, if $|x''-x'|\le r$ then 
	\begin{equation}\label{3cr}
		|f(x'')-f(x')|\le|f(x'')-f(x)|+|f(x')-f(x)|\le c|x''-x|+c|x'-x|\le 3c|x'-x|=3cr.
	\end{equation}
If $(x',f(x'))$ is not a vertex of any rectangle of our construction, then for each $n$ it is in the interior of a rectangle $Q_{n,i}'$. By selecting $(x,f(x))$ to be the endpoint of its diagonal closest to $x'$ and putting $r=|x'-x|$, \eqref{3cr} holds for every $x''$ with  $|x''-x'|\le r$. Since this estimate holds for an arbitrary small $r$, therefore $l_f(x')\le 3c$.
\end{proof}
Now we turn to the proof of Theorem \ref{HighD}.
%\begin{thm} For $m\geq 2$ there exists a continuous function $f:[0,1]^m \rightarrow [0,1]$ such that $l_{f}(x) < \infty$ for all $x \in [0,1]^m$, and $f^{-1}(y)$ is not rectifiable  for all except countably many $y \in [0,1]$. 
%\end{thm}
\begin{proof}[Proof of Theorem \ref{HighD}] Our construction of the function is based on an iterative argument on cuboids in $\mathbb{R}^{m+1}$. By a cuboid, $Q$, in $\mathbb{R}^{m+1}$, we mean a Cartesian product of an axis-parallel cube in $\mathbb{R}^m$ and an interval. For a cube, $C$, in $\mathbb{R}^m$ we denote its side-length by $l(C)$, and for a cuboid $Q =C \times I$, we define the length, $l(Q)$, by $l(Q):=l(C)$ and the height, $h(Q)$, by $h(Q):=l(I)$. We write an element in $\mathbb{R}^{m+1}$ as $(x,z)=(x_1,..,x_m,z)$. If $(x,z)$ lies on the graph of the function we are going to construct, then the first $m$ coordinates denote the domain and $z$ denotes the value of the function. We will also denote $y$ as $y=(x_2,..,x_m)$. The direction of $(0,..,0,1)$ is called the $z$-axis. The projection maps onto the $m$-dimensional plane $z=0$, the $(m-1)$-dimensional plane $x_1=z=0$ and the $z$-axis will be denoted by $\mathbb{P}_{x}$, $\mathbb{P}_{y}$ and $\mathbb{P}_{z}$ respectively.

We start with the cube $[0,1]^{m+1}$, $E_{1}=\emptyset$, and a sequence $(a_n)_{n=0}^{\infty}$ satisfying $\prod_{n=0}^{\infty}(1-a_{n})^{-1}< 2$. Assume that we have completed $n$ steps of our iterative argument and we have countably many cuboids $(Q_{n,i})_{i \in \mathbb{N}}$ in $\mathbb{R}^{m+1}$ and that $f$ is defined on $E_{n}=[0,1]^m \setminus  \mathbb{P}_{x}(\bigcup_{i=1}^{\infty}Q_{n,i})$.

Let $Q=Q_{n,i}=C\times [z,z+h]$ be such a cuboid. Let $\alpha=h(Q)/l(Q)=h/l(C)$ and let $D$ be its diagonal plane parallel to the plane $z=\alpha x_1$. We apply our iterative argument on $Q$ whose key parts are the following: 
\begin{enumerate}
	\item[\textbf{Part (1)}] We apply the Whitney decomposition theorem to $C$ and decompose it into countably many Whitney cubes $\{C_j\}_{j\in \mathbb{N}}$ such that the diameter of each cube is less than the distance to the boundary of $C$ and at least a constant times the distance to the boundary of $C$.
	\item[\textbf{Part (2)}] Let $C_j$ be a Whitney cube obtained in Part (1) with length $l(C_j)$. We partition the smaller cube with the same center as $C_j$ and length $(1-a_n)^{1/2}l(C_j)$ into $N^m$ equal cubes of length $(1-a_n)^{1/2}l(C_j)/{N}$ each, where $N$ will be chosen later. Inside each of these $N^{m}$ cubes, we consider a smaller cube with the same center and length $(1-a_n)l(C_j)/{N}$. Let these smaller cubes be denoted as $C_{j,k}$, $k\in \{1, 2, ..., N^{m}\}$. We chose $N$ large enough so that we can label $C_{j,k}$ with $N$ different labels such that
	\begin{enumerate}
		\item[(i)] For any cube with a label, $\omega$, and any point $x$ in the cube, the ball of centre $x$ and radius $2\sqrt{m}\cdot l(C_{j, k})$ intersects at most one cube of each label. Moreover, we choose $N$ to be large enough so that this ball lies in the cube $C_j$.  
		\item[(ii)] For every label, $\omega$, the union of cubes of label, $\omega$, has full $\mathbb{P}_{y}$ projection. In other words, we require that $\mathbb{P}_y(C_j)= \mathbb{P}_y( \bigcup_{k \in L_{\omega}} C_{j, k})$ where $L_{\omega}$ is the set of all indices $k$ such that  $C_{j, k}$ has label $\omega$.
	\end{enumerate}
Thus we divided $C_{j}$ into $N^{m}$ labeled cubes and narrow strips, i.e, the part of $C_{j}$ which is not covered by labeled cubes. (See Figure 3 for $m=2$). 
\begin{center}
\begin{tikzpicture}

xlabel=Subsequential Treatments over Time ,

   \newcommand\Square[1]{+(-#1,-#1) rectangle +(#1,#1)}
    \draw [ ](3,2) \Square{145pt} ;
     \draw [](-1,-2) \Square{10pt}node[pos=.5]{1};
     \draw [](0,-2) \Square{10pt}node[pos=.5]{2};
     \draw [](1,-2) \Square{10pt}node[pos=.5]{3};
     \draw [](2,-2) \Square{10pt}node[pos=.5]{4};
     \draw [](3,-2) \Square{10pt}node[pos=.5]{5};
     \draw [](4,-2) \Square{10pt}node[pos=.5]{6};
     \draw [](5,-2) \Square{10pt}node[pos=.5]{7};
     \draw [](6,-2) \Square{10pt}node[pos=.5]{8};
     \draw [](7,-2) \Square{10pt}node[pos=.5]{9};
   
   \draw [](-1,-1) \Square{10pt}node[pos=.5]{7};
     \draw [](0,-1) \Square{10pt}node[pos=.5]{8};
     \draw [](1,-1) \Square{10pt}node[pos=.5]{9};
     \draw [](2,-1) \Square{10pt}node[pos=.5]{1};
     \draw [](3,-1) \Square{10pt}node[pos=.5]{2};
     \draw [](4,-1) \Square{10pt}node[pos=.5]{3};
     \draw [](5,-1) \Square{10pt}node[pos=.5]{4};
     \draw [](6,-1) \Square{10pt}node[pos=.5]{5};
     \draw [](7,-1) \Square{10pt}node[pos=.5]{6};
     
     \draw [](-1,0) \Square{10pt}node[pos=.5]{4};
     \draw [](0,0) \Square{10pt}node[pos=.5]{5};
     \draw [](1,0) \Square{10pt}node[pos=.5]{6};
     \draw [](2,0) \Square{10pt}node[pos=.5]{7};
     \draw [](3,0) \Square{10pt}node[pos=.5]{8};
     \draw [](4,0) \Square{10pt}node[pos=.5]{9};
     \draw [](5,0) \Square{10pt}node[pos=.5]{1};
     \draw [](6,0) \Square{10pt}node[pos=.5]{2};
     \draw [](7,0) \Square{10pt}node[pos=.5]{3};
     
     \draw [](-1,1) \Square{10pt}node[pos=.5]{1};
     \draw [](0,1) \Square{10pt}node[pos=.5]{2};
     \draw [](1,1) \Square{10pt}node[pos=.5]{3};
     \draw [](2,1) \Square{10pt}node[pos=.5]{4};
     \draw [](3,1) \Square{10pt}node[pos=.5]{5};
     \draw [](4,1) \Square{10pt}node[pos=.5]{6};
     \draw [](5,1) \Square{10pt}node[pos=.5]{7};
     \draw [](6,1) \Square{10pt}node[pos=.5]{8};
     \draw [](7,1) \Square{10pt}node[pos=.5]{9};
   
   \draw [](-1,2) \Square{10pt}node[pos=.5]{7};
     \draw [](0,2) \Square{10pt}node[pos=.5]{8};
     \draw [](1,2) \Square{10pt}node[pos=.5]{9};
     \draw [](2,2) \Square{10pt}node[pos=.5]{1};
     \draw [](3,2) \Square{10pt}node[pos=.5]{2};
     \draw [](4,2) \Square{10pt}node[pos=.5]{3};
     \draw [](5,2) \Square{10pt}node[pos=.5]{4};
     \draw [](6,2) \Square{10pt}node[pos=.5]{5};
     \draw [](7,2) \Square{10pt}node[pos=.5]{6};
     
     \draw [](-1,3) \Square{10pt}node[pos=.5]{4};
     \draw [](0,3) \Square{10pt}node[pos=.5]{5};
     \draw [](1,3) \Square{10pt}node[pos=.5]{6};
     \draw [](2,3) \Square{10pt}node[pos=.5]{7};
     \draw [](3,3) \Square{10pt}node[pos=.5]{8};
     \draw [](4,3) \Square{10pt}node[pos=.5]{9};
     \draw [](5,3) \Square{10pt}node[pos=.5]{1};
     \draw [](6,3) \Square{10pt}node[pos=.5]{2};
     \draw [](7,3) \Square{10pt}node[pos=.5]{3};
     
      \draw [](-1,4) \Square{10pt}node[pos=.5]{1};
     \draw [](0,4) \Square{10pt}node[pos=.5]{2};
     \draw [](1,4) \Square{10pt}node[pos=.5]{3};
     \draw [](2,4) \Square{10pt}node[pos=.5]{4};
     \draw [](3,4) \Square{10pt}node[pos=.5]{5};
     \draw [](4,4) \Square{10pt}node[pos=.5]{6};
     \draw [](5,4) \Square{10pt}node[pos=.5]{7};
     \draw [](6,4) \Square{10pt}node[pos=.5]{8};
     \draw [](7,4) \Square{10pt}node[pos=.5]{9};
   
   \draw [](-1,5) \Square{10pt}node[pos=.5]{7};
     \draw [](0,5) \Square{10pt}node[pos=.5]{8};
     \draw [](1,5) \Square{10pt}node[pos=.5]{9};
     \draw [](2,5) \Square{10pt}node[pos=.5]{1};
     \draw [](3,5) \Square{10pt}node[pos=.5]{2};
     \draw [](4,5) \Square{10pt}node[pos=.5]{3};
     \draw [](5,5) \Square{10pt}node[pos=.5]{4};
     \draw [](6,5) \Square{10pt}node[pos=.5]{5};
     \draw [](7,5) \Square{10pt}node[pos=.5]{6};
     
     \draw [](-1,6) \Square{10pt}node[pos=.5]{4};
     \draw [](0,6) \Square{10pt}node[pos=.5]{5};
     \draw [](1,6) \Square{10pt}node[pos=.5]{6};
     \draw [](2,6) \Square{10pt}node[pos=.5]{7};
     \draw [](3,6) \Square{10pt}node[pos=.5]{8};
     \draw [](4,6) \Square{10pt}node[pos=.5]{9};
     \draw [](5,6) \Square{10pt}node[pos=.5]{1};
     \draw [](6,6) \Square{10pt}node[pos=.5]{2};
     \draw [](7,6) \Square{10pt}node[pos=.5]{3};

\draw [->,>=stealth] (-1.5,-3.6) -- (7.5,-3.6)node[below]{\footnotesize $x_1$-axis};
\draw [->,>=stealth] (-2.7,-2.5) -- (-2.7,6.5)node[left]{\footnotesize $y$-axis};

 \end{tikzpicture}    
 \end{center} 
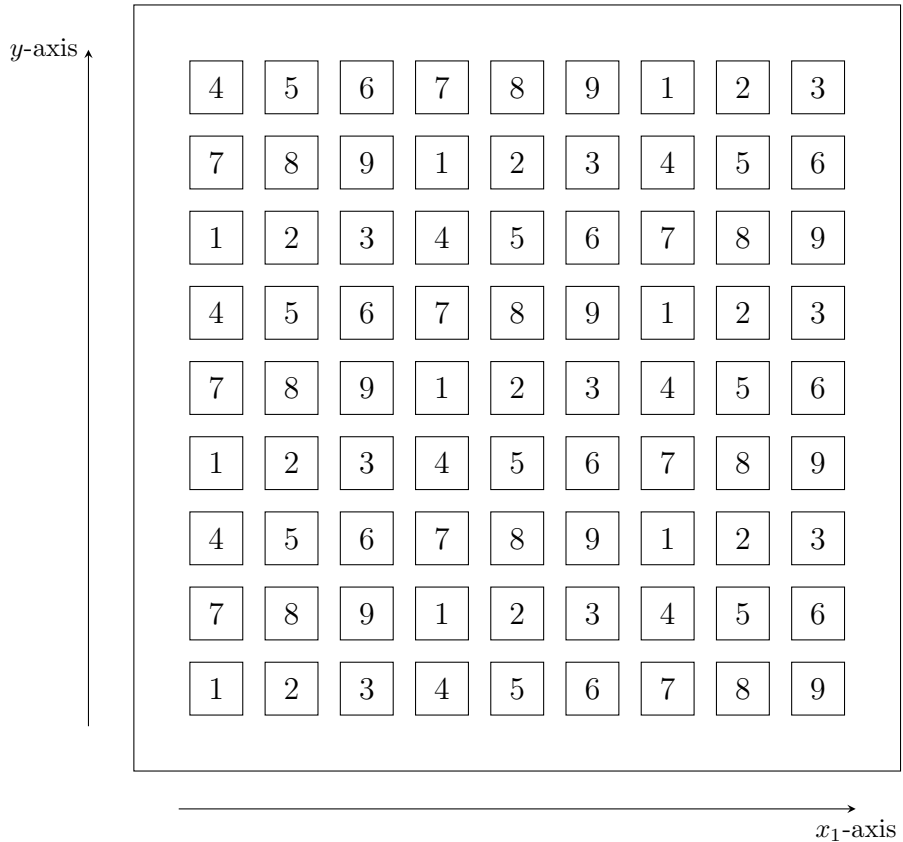
$$\textbf{Figure  4}: \text{ Decomposition of } C_{j} \text{ into cubes with 9 labels and narrow strips.} $$
	\item[\textbf{Part (3)}] Let $D_{j}\subset D$ be the intersection of $D$ and the pre-image of  $C_j$ under the projection $\mathbb{P}_x$. Consider the cuboid $R_j$ in $\mathbb{R}^{m+1}$ with sides parallel to the coordinate axes and $D_j$ as its diagonal plane. We denote $R_{j}=C_j\times [z_j,z_j+h_j].$ We call the cuboids $\{R_{j}\}_{j \in \mathbb{N}}$ the child cuboids of $Q$. (See Figure 4 for $m=2$).

\begin{center}
\begin{tikzpicture}[ line cap=round, line join=round ]
\begin{scope}[yscale = .91, xscale = .91, shift={(0.5,-1.25)}]
	% Axes

%diagonal plane
\draw[thick, black] (3.3,-10) -- (5.8,8) ;
\draw[thick] (-12,-10) -- (-8.85,8) ;
\draw[thick, black] (3.3,-10) -- (-12,-10);
\draw[thick] (5.8,8) -- (-8.85,8) ;

\draw[black] (-11.6,-10) -- (-8.43,8) ;
\draw[black] (-11.2,-10) -- (-8.03,8) ;
\draw[black] (-10.67,-10) -- (-7.703,8) ;
\draw[black] (-9.27,-10) -- (-6.503,8) ;
\draw[black] (0.4,-10) -- (3.003,8) ;
\draw[black] (2.3,-10) -- (4.9,8) ;
\draw[black] (2.8,-10) -- (5.4,8) ;
\draw[black] (1.8,-10) -- (4.4,8) ;
\end{scope}
\begin{scope}[yscale = .25, xscale = .25, shift={(-7,-7)}]
	% Axes
\draw[line width=0.5mm, black] (-11,3) -- (3,3) ;
\draw[line width=0.5mm, black] (-11,-11) -- (3,-11) ;
\draw[line width=0.5mm, black] (-11,3) -- (-11,-11) ;
\draw[line width=0.5mm, black] (3,3) -- (3,-11) ;

\draw[line width=0.5mm, black] (-8,9) -- (6,9) ;
\draw[line width=0.5mm, black, dashed] (-8,-5) -- (6,-5) ;
\draw[line width=0.5mm, black, dashed] (-8,9) -- (-8,-5) ;
\draw[line width=0.5mm, black, dashed] (6,-5) -- (6,9) ;

\draw[line width=0.5mm, black, dashed] (-11,-11) -- (-8,-5) ;

\draw[line width=0.5mm, black] (-11,3) -- (-8,9) ;
\draw[line width=0.5mm, black, dashed] (3,-11) -- (6,-5) ;
\draw[line width=0.5mm, black, dashed] (4.5,-8) -- (6,-5) ;
\draw[line width=0.5mm, black] (3,3) -- (6,9) ;

%diagonal plane
\draw[black] (3,-11) -- (6,9) ;
\draw[black] (-11,-11) -- (-8,9) ;
\end{scope}
\begin{scope}[yscale = .25/2, xscale = .25/2, shift={(-5,15)}]
	% Axes
\draw[line width=0.5mm,black] (-11,3) -- (3,3) ;
\draw[line width=0.5mm,black] (-11,-11) -- (3,-11) ;
\draw[line width=0.5mm,black] (-11,3) -- (-11,-11) ;
\draw[line width=0.5mm,black] (3,3) -- (3,-11) ;

\draw[line width=0.5mm,black] (-8,9) -- (6,9) ;
\draw[line width=0.5mm,black, dashed] (-8,-5) -- (6,-5) ;
\draw[line width=0.5mm,black, dashed] (-8,9) -- (-8,-5) ;
\draw[line width=0.5mm,black] (6,-5) -- (6,9) ;

\draw[line width=0.5mm,black, dashed] (-11,-11) -- (-8,-5) ;
\draw[line width=0.5mm,black] (-11,3) -- (-8,9) ;
\draw[line width=0.5mm,black] (3,-11) -- (6,-5) ;
\draw[line width=0.5mm,black] (3,3) -- (6,9) ;

%diagonal plane
\draw[black] (3,-11) -- (6,9) ;
\draw[black] (-11,-11) -- (-8,9) ;
\end{scope}
\begin{scope}[yscale = .25/2, xscale = .25/2, shift={(-19,15)}]
	% Axes
\draw[line width=0.5mm,black, ] (-11,3) -- (3,3) ;
\draw[line width=0.5mm,black] (-11,-11) -- (3,-11) ;
\draw[line width=0.5mm,black] (-11,3) -- (-11,-11) ;

\draw[line width=0.5mm,black] (-8,9) -- (6,9) ;
\draw[line width=0.5mm,black,dashed] (-8,-5) -- (6,-5) ;
\draw[line width=0.5mm,black,dashed] (-8,9) -- (-8,-5) ;

\draw[line width=0.5mm,black, dashed] (-11,-11) -- (-8,-5) ;
\draw[line width=0.5mm, black] (-11,3) -- (-8,9) ;
\draw[line width=0.5mm,black, dashed] (3,-11) -- (6,-5) ;
\draw[line width=0.5mm,black] (3,3) -- (6,9) ;

%diagonal plane
\draw[black] (3,-11) -- (6,9) ;
\draw[black] (-11,-11) -- (-8,9) ;
\end{scope}

\begin{scope}[yscale = .25/2, xscale = .25/2, shift={(-28,-45)}]
	% Axes
\draw[line width=0.5mm,black, ] (-11,3) -- (3,3) ;
\draw[line width=0.5mm,black] (-11,-11) -- (3,-11) ;
\draw[line width=0.5mm,black] (-11,3) -- (-11,-11) ;

\draw[line width=0.5mm,black] (-8,9) -- (6,9) ;
\draw[line width=0.5mm,black,dashed] (-8,-5) -- (6,-5) ;
\draw[line width=0.5mm,black,dashed] (-8,9) -- (-8,-5) ;

\draw[line width=0.5mm,black, dashed] (-11,-11) -- (-8,-5) ;
\draw[line width=0.5mm, black] (-11,3) -- (-8,9) ;
\draw[line width=0.5mm,black, dashed] (3,-11) -- (6,-5) ;
\draw[line width=0.5mm,black] (3,3) -- (6,9) ;

%diagonal plane
\draw[black] (3,-11) -- (6,9) ;
\draw[black] (-11,-11) -- (-8,9) ;
\end{scope}

\begin{scope}[yscale = .25/2, xscale = .25/2, shift={(-14,-45)}]
	% Axes
\draw[line width=0.5mm,black] (-11,3) -- (3,3) ;
\draw[line width=0.5mm,black] (-11,-11) -- (3,-11) ;
\draw[line width=0.5mm,black] (-11,3) -- (-11,-11) ;
\draw[line width=0.5mm,black] (3,3) -- (3,-11) ;

\draw[line width=0.5mm,black] (-8,9) -- (6,9) ;
\draw[line width=0.5mm,black, dashed] (-8,-5) -- (6,-5) ;
\draw[line width=0.5mm,black, dashed] (-8,9) -- (-8,-5) ;
\draw[line width=0.5mm,black] (6,-5) -- (6,9) ;

\draw[line width=0.5mm,black, dashed] (-11,-11) -- (-8,-5) ;
\draw[line width=0.5mm,black] (-11,3) -- (-8,9) ;
\draw[line width=0.5mm,black] (3,-11) -- (6,-5) ;
\draw[line width=0.5mm,black] (3,3) -- (6,9) ;

%diagonal plane
\draw[black] (3,-11) -- (6,9) ;
\draw[black] (-11,-11) -- (-8,9) ;
\end{scope}

\begin{scope}[yscale = .25/2, xscale = .25/2, shift={(3,-25)}]
	% Axes
\draw[line width=0.5mm,black] (-11,3) -- (3,3) ;
\draw[line width=0.5mm,black] (-11,-11) -- (3,-11) ;
\draw[line width=0.5mm,black] (-11,3) -- (-11,-11) ;
\draw[line width=0.5mm,black] (3,3) -- (3,-11) ;

\draw[line width=0.5mm,black] (-8,9) -- (6,9) ;
\draw[line width=0.5mm,black, dashed] (-8,-5) -- (6,-5) ;
\draw[line width=0.5mm,black, dashed] (-8,9) -- (-8,-5) ;
\draw[line width=0.5mm,black] (6,-5) -- (6,9) ;

\draw[line width=0.5mm,black, dashed] (-11,-11) -- (-8,-5) ;
\draw[line width=0.5mm,black] (-11,3) -- (-8,9) ;
\draw[line width=0.5mm,black] (3,-11) -- (6,-5) ;
\draw[line width=0.5mm,black] (3,3) -- (6,9) ;

%diagonal plane
\draw[black] (3,-11) -- (6,9) ;
\draw[black] (-11,-11) -- (-8,9) ;
\end{scope}
\begin{scope}[yscale = .25/2, xscale = .25/2, shift={(6,-5)}]
	% Axes
\draw[line width=0.5mm,black] (-11,3) -- (3,3) ;
\draw[line width=0.5mm,black] (-11,-11) -- (3,-11) ;
\draw[line width=0.5mm,black] (-11,3) -- (-11,-11) ;
\draw[line width=0.5mm,black] (3,3) -- (3,-11) ;

\draw[line width=0.5mm,black] (-8,9) -- (6,9) ;
\draw[line width=0.5mm,black, dashed] (-8,-5) -- (6,-5) ;
\draw[line width=0.5mm,black, dashed] (-8,9) -- (-8,-5) ;
\draw[line width=0.5mm,black] (6,-5) -- (6,9) ;

\draw[line width=0.5mm,black, dashed] (-11,-11) -- (-8,-5) ;
\draw[line width=0.5mm,black] (-11,3) -- (-8,9) ;
\draw[line width=0.5mm,black] (3,-11) -- (6,-5) ;
\draw[line width=0.5mm,black] (3,3) -- (6,9) ;

%diagonal plane
\draw[black] (3,-11) -- (6,9) ;
\draw[black] (-11,-11) -- (-8,9) ;
\end{scope}
\begin{scope}[yscale = .25/2, xscale = .25/2, shift={(-36,-5)}]
	% Axes
\draw[line width=0.5mm,black] (-11,3) -- (3,3) ;
\draw[line width=0.5mm,black] (-11,-11) -- (3,-11) ;
\draw[line width=0.5mm,black] (-11,3) -- (-11,-11) ;
\draw[line width=0.5mm,black] (3,3) -- (3,-11) ;

\draw[line width=0.5mm,black] (-8,9) -- (6,9) ;
\draw[line width=0.5mm,black, dashed] (-8,-5) -- (6,-5) ;
\draw[line width=0.5mm,black, dashed] (-8,9) -- (-8,-5) ;

\draw[line width=0.5mm,black, dashed] (-11,-11) -- (-8,-5) ;
\draw[line width=0.5mm,black] (-11,3) -- (-8,9) ;
\draw[line width=0.5mm,black, dashed] (3,-11) -- (6,-5) ;
\draw[line width=0.5mm,black] (3,3) -- (6,9) ;

%diagonal plane
\draw[black] (3,-11) -- (6,9) ;
\draw[black] (-11,-11) -- (-8,9) ;
\end{scope}
	
\begin{scope}[yscale = .25/2, xscale = .25/2, shift={(-39,-25)}]
	% Axes
\draw[line width=0.5mm,black] (-11,3) -- (3,3) ;
\draw[line width=0.5mm,black] (-11,-11) -- (3,-11) ;
\draw[line width=0.5mm,black] (-11,3) -- (-11,-11) ;
\draw[line width=0.5mm,black] (3,3) -- (3,-11) ;

\draw[line width=0.5mm,black] (-8,9) -- (6,9) ;
\draw[line width=0.5mm,black, dashed] (-8,-5) -- (6,-5) ;
\draw[line width=0.5mm,black, dashed] (-8,9) -- (-8,-5) ;
\draw[line width=0.5mm,black, dashed] (6,9) -- (6,-5) ;

\draw[line width=0.5mm,black, dashed] (-11,-11) -- (-8,-5) ;
\draw[line width=0.5mm,black] (-11,3) -- (-8,9) ;
\draw[line width=0.5mm,black, dashed] (3,-11) -- (6,-5) ;
\draw[line width=0.5mm,black] (3,3) -- (6,9) ;

%diagonal plane
\draw[black] (3,-11) -- (6,9) ;
\draw[black] (-11,-11) -- (-8,9) ;
\end{scope}
\begin{scope}[yscale = .25/2, xscale = .25/2, shift={(-33,15)}]

\draw[black] (3,-11) -- (6,9) ;
\draw[black] (-11,-11) -- (-8,9) ;
\draw[black] (-8,9)--(6,9);
\draw[black] (-11,-11)--(3,-11);
\end{scope}

\begin{scope}[yscale = .25/2, xscale = .25/2, shift={(9,15)}]

\draw[black] (3,-11) -- (6,9) ;
\draw[black] (-11,-11) -- (-8,9) ;
\draw[black] (-8,9)--(6,9);
\draw[black] (-11,-11)--(3,-11);
\end{scope}
\begin{scope}[yscale = .25/2, xscale = .25/2, shift={(-42,-45)}]

\draw[black] (3,-11) -- (6,9) ;
\draw[black] (-11,-11) -- (-8,9) ;
\draw[black] (-8,9)--(6,9);
\draw[black] (-11,-11)--(3,-11);
\end{scope}
\begin{scope}[yscale = .25/2, xscale = .25/2, shift={(0,-45)}]

\draw[black] (3,-11) -- (6,9) ;
\draw[black] (-11,-11) -- (-8,9) ;
\draw[black] (-8,9)--(6,9);
\draw[black] (-11,-11)--(3,-11);
\end{scope}
\begin{scope}[yscale = .75, xscale = .75, shift={(-0.5,-.8)}]

\draw[black] (2.73,-12.8) -- (5.9,9) ;
\draw[black] (-11.05,-12.8) -- (-7.55,9) ;

\draw[black] (3 ,4.8) -- (3.6,9) ;
\draw[black] (1.9 ,4.8) -- (2.5,9) ;

\draw[black] (.7, 4.8) -- (1.34,9) ;
\draw[black] (-.55, 4.85) -- (.1,9) ;
\draw[black] (-10,7)--(7.8,7);

\draw[black] (-10.15,6)--(7.7,6);
\draw[black] (-9.8,8)--(8,8);
\draw[black] (-9.7,8.5)--(8.1,8.5);
\draw[black] (-9.9,7.5)--(7.85,7.5);
\draw[black] (-1.65 ,4.8) -- (-1,9) ;
\draw[black] (-6.3,4.8) -- (-5.7,9) ;
\draw[black] (-2.7 ,4.85) -- (-2.1, 9) ;
\draw[black] (-3.9, 4.85) -- (-3.3,9) ;
\draw[black] (-5, 4.8) -- (-4.4, 9) ;

\draw[black] (-13.15,-11)--(5.3,-11);
\draw[black] (-6.5, 7) -- (-6.2,9) ;
\draw[black] (-5.3, 7) -- (-5,9) ;
\draw[black] (-4.1, 7) -- (-3.8,9) ;
\draw[black] (-3, 7) -- (-2.7,9) ;
\draw[black] (-1.8, 7) -- (-1.5,9) ;
\draw[black] (-.8, 7) -- (-.5,9) ;
\draw[black] (.5, 7) -- (.8,9) ;
\draw[black] (1.6, 7) -- (1.9,9) ;
\draw[black] (2.7, 7) -- (3,9) ;
\draw[black] (3.9, 7) -- (4.2,9) ;
\draw[black] (5.1, 7) -- (5.4,9) ;

\end{scope}
\begin{scope}[yscale = .75, xscale = .75, shift={(-3.1,-18.1)}]

\draw[black] (3 ,4.5) -- (3.6,8.85) ;
\draw[black] (-6.3,4.5) -- (-5.7,8.87) ;
\draw[black] (-1.65 ,4.4) -- (-1,8.85) ;
\draw[black] (.7, 4.5) -- (1.34,8.9) ;
\draw[black] (-3.9, 4.4) -- (-3.3,8.9) ;
\end{scope}
\begin{scope}[yscale = .75, xscale = .75, shift={(-3.5,-20.5)}]

\draw[black] (-6.5, 6.85) -- (-6.2,8.7);
\draw[black] (-5.3, 6.85) -- (-5,8.7);
\draw[black] (-4.1, 6.85) -- (-3.8,8.7);
\draw[black] (-4.6, 6.9) -- (-4.051,11.2);
\draw[black] (-2.4, 6.9) -- (-1.8,11.2);
\draw[black] (-.1, 6.9) -- (.5,11.2);
\draw[black] (2.2, 6.9) -- (2.8,11.2);
\draw[black] (-3, 6.85) -- (-2.7,8.7);
\draw[black] (-1.82, 6.9) -- (-1.6,8.7);
\draw[black] (-.6, 6.9) -- (-.33,8.7);
\draw[black] (.5, 6.9) -- (.8,8.7);
\draw[black] (1.6, 6.9) -- (1.9,8.7);
\draw[black] (2.8, 6.9) -- (3.1,8.7);
\draw[black] (3.95, 6.9) -- (4.25,8.7);
\draw[black] (5.1, 6.9) -- (5.4,8.7);
\end{scope}
\begin{scope}[yscale = .75, xscale = .75, shift={(-.3,.7)}]
9);
\draw[black] (-12.85,-10.04)--(-8.6,-10.04);
\draw[black] (-12.3,-6.73)--(-8,-6.73);
\draw[black] (-11.7,-3.37)--(-7.4,-3.36);
\draw[black] (-11.15,-0.051)--(-6.9,-0.051);
\draw[black] (-10.56,3.29)--(-6.6,3.29);
\end{scope}
\begin{scope}[yscale = .75, xscale = .75, shift={(-5.3,-19)}]
9);
\draw[black] (-8.6,5.8)--(9.92,5.8);
\draw[black] (-8.1,8.4)--(10.26,8.4);
\draw[black] (-8.43,6.75)--(10.1,6.75);
\draw[black] (-8.5,6.3)--(9.99,6.3);

\end{scope}

\begin{scope}[yscale = .75, xscale = .75, shift={(13.3,.7)}]
9);
\draw[black] (-12.8,-10.04)--(-8.1,-10.04);
\draw[black] (-12.6,-8.4)--(-7.87,-8.4);
\draw[black] (-26.23,-8.4)--(-21.92,-8.4);
\draw[black] (-25.6,-5)--(-21.45,-5);
\draw[black] (-12,-5)--(-7.4,-5);
\draw[black] (-12.3,-6.72)--(-7.6,-6.73);
\draw[black] (-11.8,-3.37)--(-7.2,-3.37);
\draw[black] (-11.3,-0.04)--(-6.7,-0.051);
\draw[black] (-11.5,-1.82)--(-6.92,-1.82);
\draw[black] (-25,-1.82)--(-20.9,-1.82);
\draw[black] (-10.75,3.3)--(-6.25,3.285);
\draw[black] (-24.4, 1.65)--(-20.4, 1.65);
\draw[black] (-11, 1.65)--(-6.45, 1.65);

\draw[black] (-8.25, 5)--(-6.05, 5);
\draw[black] (-23.85, 5)--(-21.8, 5);

\draw[black] (-8.4, 3.9)--(-6.2, 3.9);
\draw[black] (-24.1, 3.9)--(-21.9, 3.9);

\draw[black] (-8.6, 2.5)--(-6.35, 2.5);
\draw[black] (-24.3, 2.5)--(-22.2, 2.5);

\draw[black] (-8.9, .8)--(-6.54, .8);
\draw[black] (-24.6, .8)--(-22.4, .8);

\draw[black] (-9.15, -1)--(-6.9, -1);
\draw[black] (-24.9, -1)--(-22.7, -1);

\draw[black] (-9.35, -2.55)--(-7, -2.55);
\draw[black] (-25.2, -2.55)--(-23, -2.55);

\draw[black] (-9.6, -4.15)--(-7.3, -4.15);
\draw[black] (-25.5, -4.15)--(-23.2, -4.15);

\draw[black] (-9.8, -5.8)--(-7.46, -5.8);
\draw[black] (-25.7, -5.8)--(-23.5, -5.8);

\draw[black] (-10.08, -7.5)--(-7.7, -7.5);
\draw[black] (-26, -7.5)--(-23.8, -7.5);

\draw[black] (-10.3, -9.2)--(-7.95, -9.2);
\draw[black] (-26.3, -9.2)--(-24, -9.2);

\draw[black] (-10.5, -10.7)--(-8.2, -10.7);
\draw[black] (-26.6, -10.7)--(-24.3, -10.7);

\draw[black] (-10.7, -11.9)--(-8.35, -11.9);
\draw[black] (-26.8, -11.9)--(-24.5, -11.9);

\end{scope}
\end{tikzpicture} 
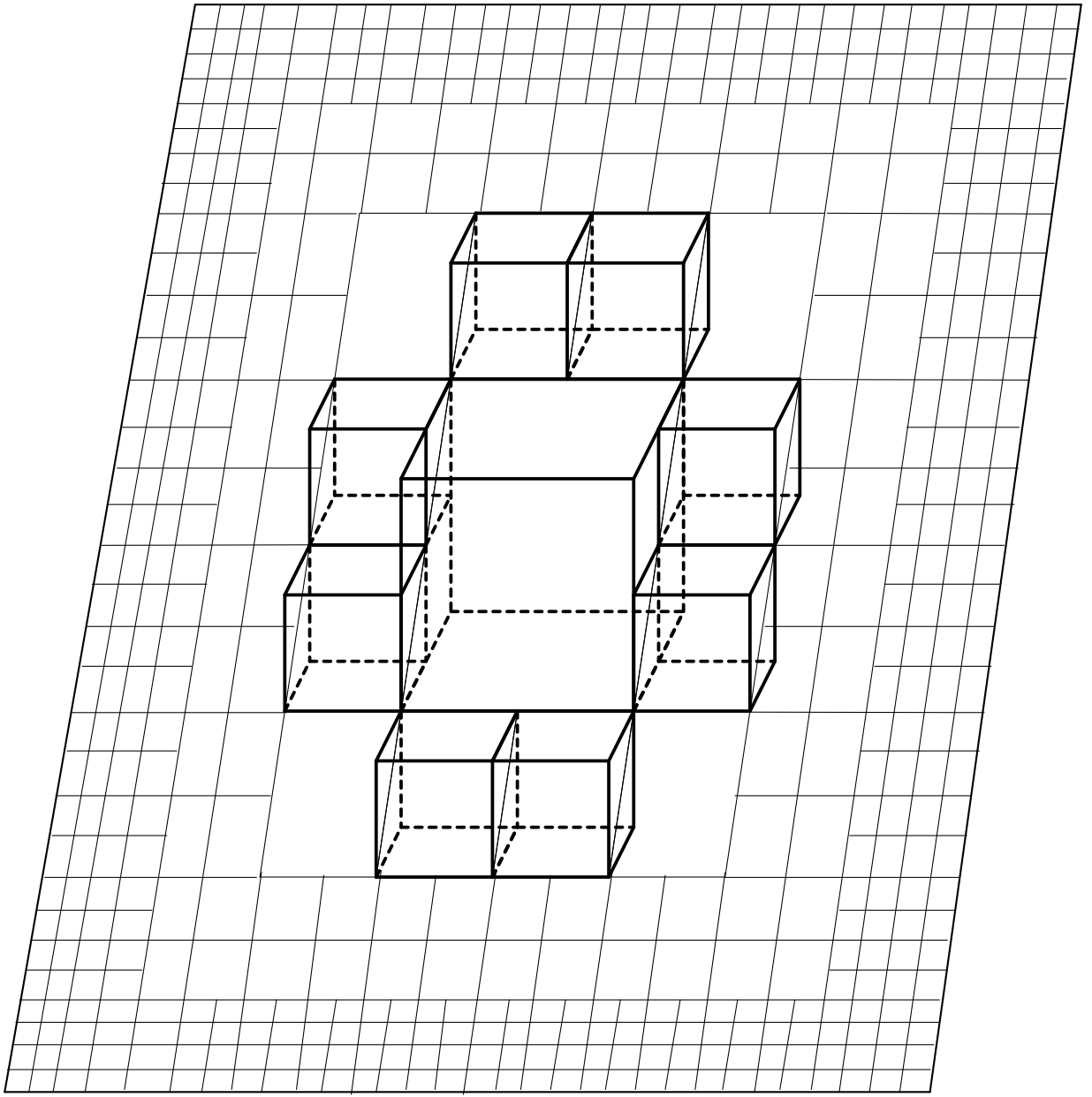
    $$\textbf{Figure 4}:  \text{ Construction of child cuboids of Q}.$$
\end{center}
	\item[\textbf{Part (4)}] Divide $[z_j,z_j+h_j]$ into $N$ equal parts and to each of these intervals assign a label from the $N$ labels in Part 2 so that no two intervals  have the same label.  Let the interval with label, $\omega$, be $I_{j, \omega}.$ If the label of a cube $C_{j,k}$ is also $\omega$, then we consider the corresponding cuboid
	\begin{align*}
		S_{j,k}=C_{j,k}\times I_{j,\omega}.
	\end{align*}
	\item[\textbf{Part (5)}] Let $D_{j,k}$ be the diagonal plane of $S_{j,k}$ parallel to the $(1-a_n)z=\alpha x_1$ plane. On the boundary of $C_{j,k}$, we define $f$ so that its graph coincides with the boundary of $D_{j,k}$, i.e if $x$ lies on the boundary of $C_{j,k}$, then  $(x,f(x))$ lies on $D_{j,k}$. Similarly, we define the graph of $f$ over the boundary of $C_{j}$ to be the boundary of $D_j$. We then extend $f$ piece-wise linearly on the narrow strips defined in Part (2) so that it is continuous, and imitating the construction of Theorem \ref{OneD}, we define $f$ to be constant near the top and bottom faces of $S_{j,k}$ so that the Hausdorff dimension of each of the level sets containing these faces is $m$. This will ensure that the level sets associated to the top and bottom faces are not $(m-1)$-rectifiable.
	\item[\textbf{Part (6)}] The countable collection of cuboids thus obtained, i.e, $\{S_{j,k}:Q=Q_{n,i},  i, j \in \mathbb{N}, k\in \{1,2..,N^{m}\}\}$ will form our new collection of cuboids. We re-enumerate them as $\{Q_{n+1,j}\}_{j \in \mathbb{N}}.$ \par
This completes the $(n+1)^{th}$ step. 
\end{enumerate}

For $x\in [0,1]^{m}\setminus \bigcup_{n=1}^{\infty}E_n$, we define the point $(x, f(x))$ to be the intersection of all the cuboids $Q_{n,i}$ whose projection onto the plane $z=0$ contains $x$. Thus the function is continuous for every $[0,1]^{m}$.
We claim that the function thus obtained has $l_{f}(x) <\infty$ and each level set is not rectifiable. \par 

Clearly if $(x,f(x))$ lies on the piece-wise linear parts of the graph of $f$ then $l_{f}(x)<\infty.$ It remains to prove so for $(x,f(x))$ which lie either (1) on the boundary of a diagonal plane or (2) in infinitely many cuboids $Q_{n,i}$. \par 

We observe that for each $n \in \mathbb{N}$, the  height/length ratio of the cuboids, $Q_{n,i}$, increases by a factor of $1/(1-a_{n})$ after the $(n+1)$th step. Since we start with the cube $[0,1]^{m+1}$, the height/length ratio for all the cuboids is bounded by $\prod_{n=0}^{\infty}(1-a_{n})^{-1}$. We recall that we've chosen $(a_n)_{n=0}^{\infty}$ such that $\prod_{n=0}^{\infty}(1-a_{n})^{-1}< 2$.

Now for the first case, let $(x,f(x))$ lie on the boundary of the diagonal plane of a cuboid $Q=Q_{n,i}$ and let $(y,f(y)) \in Q$. The graph of the function inside $Q$ lies in its child cuboids. Let $(y,f(y)) \in R_{j}$, where $R_{j}$ is a child cuboid of $ Q$ and let $(w,f(w))$ be a point on the diagonal plane of $R_{j}$ closest to $(x,f(x))$. By the properties of the Whitney decomposition and the diagonal planes, the length of the interval $I_{j}=\mathbb{P}_{z}(R_{j})$ is less than the distance between the boundary of $I=\mathbb{P}_{z}(Q)$ and $I_j$. Thus the height of the cuboid $R_{j}$ is less than $|f(w)-f(x)|.$ This implies $|f(y)-f(w)| \leq |f(w)-f(x)|$ and 
	\begin{align*}
		\frac{|f(y)-f(x)|}{|y-x|} \leq  \frac{|f(y)-f(w)|+|f(w)-f(x)|}{|y-x|} \leq  \frac{2|f(w)-f(x)|}{|y-x|}.
	\end{align*}
Since we have taken the closest vertex to $(x,f(x))$ we have $|y-x|\geq |w-x|$. Thus 
	\begin{align*}
		\frac{2|f(w)-f(x)|}{|y-x|} \leq \frac{2|f(w)-f(x)|}{|w-x|} < c,
	\end{align*}
for some constant $c$. Thus $l_{f}(x) < \infty$. 

%Next we estimate $l_{f}(y).$  Let $(x',f(x'))$ be the nearest point to $(y,f(y))$ on the diagonal plane of the cuboid and let $r=|y-x'|$. Thus $B(y,r) \subset \mathbb{P}_x(Q)$ and for $v\in B(y,r)$ we have
%\begin{align*}
%    |f(y)-f(v)|  & \leq |f(y)-f(x')| + |f(x')-f(v)| \\
%    & < c|y-x'|+c|x'-v| \\ 
%    & < c|y-x'|+c|x'-y|+c|y-v| \leq 3cr
%\end{align*}
%If $(y,f(y))$ lies on the piece-wise linear parts or on the diagonal plane of some cuboid $Q_{n,i}$ then we have already proven that $l_{f}(y) < \infty$. 

For the second case, we estimate $l_{f}(y)$ where $(y,f(y))$ lies inside infinitely many cuboids $Q_n:=Q_{n,i}$ of arbitrarily small length. Let $(x_n,f(x_n))$ be a point on the diagonal plane of $Q_n$ such that if $s_n:=|y-x_n|$, then $B(y, s_n) \subset \mathbb{P}_x(Q_n)$. From above, we know that  if $(x_n, f(x_n))$ lies on the boundary of $Q_n$, then $\frac{|f(x_n)-f(v)|}{|x_n-v|}<c$ for every $v \in \mathbb{P}_x(Q_n)$. Thus for $v\in B(y,s_n) \subset \mathbb{P}_x(Q_n)$ we have
\begin{align*}
    |f(y)-f(v)|  & \leq |f(y)-f(x_n)| + |f(x_n)-f(v)| \\
    & < c|y-x_n|+c|x_n-v| \\ 
    & < c|y-x_n|+c|x_n-y|+c|y-v| \leq 3cs_n.
\end{align*}
Therefore, $l_{f}(y)< \infty$ and we have shown that the lower-scaled oscillation of $f$ is bounded in every case.

It remains to prove that the level sets are not rectifiable. 
Define 
	\begin{align*}
		F=\{(x,f(x)) : (x,f(x)) \text{ lies in infinitely many cuboids } Q_{n,i}\}
	\end{align*}
	 and let $F_z=f^{-1}(z) \cap F$. By property (ii) in Part (2) of the cuboid construction, we see that $F_{z}$ has full $\mathbb{P}_y$-projection. Thus ${\mathcal{H}}^{m-1}(F_z)>0$. We claim that $F_{z}$ is unrectifiable and thus $f^{-1}(z)$ is unrectifiable.
	 
Consider $x \in F_z$. Then there exists some $N$ such that $n>N$ implies  that  $x \in Q_{n}$ for some cuboid, $Q_n:=Q_{n, i}$, constructed at step $n$. By the construction of  $Q_{n}$ and $Q_{n+1}$ (Part (2)), we can say that for $r_n= \sqrt{m}\cdot l(Q_{n+1})$, $\mathbb{P}_x(Q_{n+1}) \subset B(x,r_n) \subset B(x,2r_n) \subset \mathbb{P}_x(Q_n)$. In addition, the labeling argument of Part (4) and Part (5) imply that
	\begin{align*}
		F_{z} \cap B(x,2r_n)=F_{z} \cap B(x,r_n). 
	\end{align*}
Thus,
	\begin{align*}
		{D}^{m-1}(F'_{z},x,2r_n)&=\frac{{{\mathcal{H}}^{m-1}(F'_{z}\cap B(x,2r_n))}}{(4r_n)^{m-1}}\\
		&=2^{-m+1}\frac{{{\mathcal{H}}^{m-1}(F'_{z}\cap B(x,r_n))}}{(2r_n)^{m-1}} = 2^{-m+1}{D}^{m-1}(F'_{z},x,r_n) .
	\end{align*}
Since this holds for  a sequence of scales, $r_n$, that converges to zero, both densities, ${D}^{m-1}(F'_{z},x,2r_n)$ and ${D}^{m-1}(F'_{z},x,r_n)$, cannot converge to 1. Therefore, the Hausdorff density of $F_z$ is not 1 at any point in $F_z$ which implies that $F_z$ is not rectifiable by Theorem \ref{densitythm}. This proves that $f^{-1}(z)$ is unrectifiable for all $z \in [0,1]\setminus E$, where $E$ is set of the level sets intersecting the top and bottom faces of the cuboids $\{Q_{n,i}\}_{n,i \in \mathbb{N}}.$ Finally, according to Part (5) of the cuboid construction, for every $z \in E$, $f^{-1}(z)$ has dimension $m$.

\end{proof}

%%%%%%

\bibliography{ACWBIB} 
\bibliographystyle{plain}

\end{document}